\documentclass{article}
\usepackage{amsmath,amsthm,amssymb,graphicx}

\newcommand{\ga}{\alpha}
\newcommand{\gb}{\beta}

\newcommand{\gw}{\omega}

\newcommand{\gs}{\sigma}

\newcommand{\supp}{\mathrm{supp}}

\newcommand{\dom}{\mathrm{dom}}

\newcommand{\cantor}{2^\gw}

\newcommand{\power}{\mathcal{P}}
\newcommand{\vareps}{\varepsilon}

\newtheorem{theorem}{Theorem}[section]

\newtheorem{fact}[theorem]{Fact}
\newtheorem{proposition}[theorem]{Proposition}

\theoremstyle{definition}
\newtheorem{definition}[theorem]{Definition}

\title{The Boolean prime ideal theorem and Vitali sets\footnote{2020 AMS subject classification 03E25, 22F05. Keywords: Boolean Prime Ideal Theorem, Axiom of Dependent Choices, intersection model}}

\author{
	Jind{\v r}ich Zapletal\\
	University of Florida\\
	zapletal@ufl.edu}

\begin{document}
\maketitle

\begin{abstract}
	If ZFC is consistent then so is the theory ZF+DC+BPI+there is no Vitali set. The proof uses intersection models of choiceless set theory. 
\end{abstract}

\section{Introduction}

The Boolean Prime Ideal theorem (BPI) is the statement that every Boolean algebra carries an ultrafilter. This is a prominent consequence of the Axiom of Choice (AC); it is equivalent, among other things, to the statement that every consistent first order theory has a consistent completion. It is not equivalent to the Axiom of Choice, as a classical theorem of Halpern and L{\' e}vy \cite{halpern:bpi} showed. Later, Blass \cite{blass:bpi} showed that BPI in permutation models is connected with the Ramsey properties of structures used to generate the models.

Since the methods of \cite{halpern:bpi} and \cite{blass:bpi} necessarily imply a violation of the Axiom of Dependent Choices (DC), it is natural to ask whether the conjunction of the BPI and DC is equivalent to AC. \cite{pincus:original} found a model which satisfies BPI+DC, but in which the Axiom of Choice for collections of cardinality $\aleph_1$ fails. That claim remains beyond the ability of the present writer to check, but more recently \cite{ransom:symmetric} found such a model using a rather standard permutation technique. In all similar models in the current literature, e.g. \cite{schilhan:ordering}, the reals are well-ordered. To rectify this shortcoming, I show in this paper that ZF+BPI+DC does not imply the existence of Vitali sets.

\begin{theorem}
	\label{maintheorem}
	If ZFC is consistent then so is the theory ZF+DC+BPI+there is no Vitali set.
\end{theorem}

\noindent The approach to the proof of Theorem~\ref{maintheorem} is quite different from the standard symmetric or permutation model constructions; it uses intersection models of \cite{z:geometric} instead. This in fact simplifies the argument considerably and brings it closer to the heart of a traditional forcing practitioner. The following is a brief description of the model.

\begin{definition}
	$\mathbb{E}_0$ is the equivalence relation on $\cantor$ connecting two binary sequences if they agree on all but finitely many positions.
\end{definition}

\noindent As is customary in descriptive set theory, I will call the $\mathbb{E}_0$ relation the \emph{Vitali equivalence relation}; a \emph{Vitali set} is a subset of $\cantor$ intersecting each $\mathbb{E}_0$-class in exactly one element.

\begin{definition}
	The \emph{$\mathbb{E}_0$}-forcing is the poset of all Borel subsets of $\cantor$ modulo the $\gs$-ideal generated by Borel partial $\mathbb{E}_0$-selectors, ordered by inclusion. Its generic object will be viewed as a function in $\cantor$.
\end{definition}

\begin{definition}
	$\mathbb{P}$ is the countable support product of $\gw_2$ many copies of the $\mathbb{E}_0$-forcing. Its generic object will be viewed as a function $c\colon \gw_2\times\gw\to 2$.
\end{definition}

\noindent Now, it is possible to describe the model for Theorem~\ref{maintheorem}. Start with a model $V$ of ZFC+CH. Let $c\colon\gw_2\times\gw\to 2$ be a function $\mathbb{P}$-generic over $V$. For each $n\in\gw$, write $c_n\colon \gw_2\times\gw\to 2$ for the function defined by $c_n(\ga, m)=c(\ga, n+m)$. Write $W=\bigcap_nV[c_n]$; then, the model $W$ satisfies ZF+DC+BPI+there are no Vitali sets.

This paper uses the standard set theoretic notation of \cite{jech:newset}. The work contained herein was partly supported by National Science Foundation grant  DMS 2348371.

\section{General theorems on intersection models}

In this section, I gather the relevant general results on intersection models. I will deal with only a certain rich subclass of the more general setup of \cite[Section 4.2]{z:geometric}.

\begin{definition}
	\label{maindefinition}
	Let $\langle \mathbb{B}_n\colon n\in\gw\rangle$ be a sequence of complete Boolean algebras such that for every $n\in\gw$, $\mathbb{B}_{n+1}$ is a complete subalgebra of $\mathbb{B}_n$. Let $G\subset\mathbb{B}_0$ be a generic filter. The \emph{associated intersection model} $M$ is the model $\bigcap_n V[G\cap\mathbb{B}_n]$.
\end{definition} 

\noindent Everywhere below, the symbols $\mathbb{B}_n$, $G$, and $M$ keep the meaning from Definition~\ref{maindefinition}. $V$ denotes the ground model.

\begin{proposition}
	\textnormal{\cite[Theorem 4.2.9]{z:geometric}} The intersection model is a transitive model of ZF.
\end{proposition}

\begin{proof}
	The model $M$ is an intersection of countably many transitive classes containing all ordinals, so it is transitive and contains all ordinals. For the same reason, it is closed under G{\" o}del functions. It is also universal in the sense that for every subset $a\subset M$ there is a set $b\in M$ such that $a\subseteq b$ holds: this occurs because for every ordinal $\lambda$, the set $b_\lambda$ of all elements of $M$ of set theoretic rank smaller than $\lambda$ belongs to $V[G\cap\mathbb{B}_n]$, so $b_\lambda\in M$ as desired. By \cite[Theorem 13.9]{jech:newset}, $M$ is a model of ZF.
\end{proof}

\begin{proposition}
	\label{dcproposition}
	\textnormal{\cite[Electronic supplement]{z:geometric}}
	The intersection model satisfies DC.
\end{proposition}

\begin{proof}
	Let $\langle R, \leq\rangle$ be a partial order in the intersection model without minimal elements; I must produce an infinite descending sequence in the intersection model.

 Let $\preceq$ in $V$ be a well-ordering of a large part of the ground model. By recursion on $n\in\gw$ build a descending sequence $\langle r_n\colon n\in\gw\rangle$ of conditions in $R$ in the following way: $r_0\in R$ is arbitrary, and for each $n\in\gw$, $r_{n+1}=\gs_{n+1}/G$ where $\tau_{n+1}$ is the $\preceq$-least $\mathbb{B}_{n+1}$-name such that $\tau_{n+1}/G< r_n$. Note that for each $m\in\gw$, the sequence $\langle r_n\colon n\geq m\rangle$ can be recovered by the same recursion process in $V[G\cap\mathbb{B}_m]$ starting from $r_m$. It follows that the descending sequence $\langle r_n\colon n\in\gw\rangle$ belongs to the intersection model as required.
\end{proof}

\begin{proposition}
	The intersection model $M$ is equal to $V(a)[b]$ where
	
	\begin{enumerate}
		\item $a=\mathbb{B}_0^\gw\cap M$;
		\item $b=\{\bar p\in a\colon \forall^\infty n\in\gw\ \bar p(n)\in\mathbb{B}_n\cap G\}$.
	\end{enumerate}
\end{proposition}

\begin{proof}
	It is quite clear that $a, b\in M$, so $V(a)[b]\subseteq M$ holds. For the opposite inclusion, first show that $V(a)$ contains all $\gw$-sequences of elements of the ground model which belong to $M$. The reason is that the range of any such sequence $f$ is covered by a set $c$ of cardinality $|\mathbb{B}_0|$ in the ground model, and then $f\in c^\gw$ is in fact a composition of a ground model bijection between $c$ and $\mathbb{B}$ with some element of $\mathbb{B}^\gw\cap M$.
	
	For a formula $\phi$ of the $\mathbb{B}_n$-language in the ground model, write $|\phi|_n$ for its Boolean value. In $V(a)[b]$, consider the class $C$ of all $\gw$-sequences $\bar\gs$ such that for each $n\in\gw$, $\bar\gs(n)$ is a ground model $\mathbb{B}_n$-name, and $\langle |\bar \gs(n)=\bar\gs(n+1)|_n\colon n\in\gw\rangle\in b$ holds.
	
	Consider also the binary relations $\approx$ and $\vareps$ on $C$ defined by 
	
	\begin{itemize}
		\item $\bar\gs\approx\bar\tau$ if $\langle |\bar\gs(n)=\bar\tau(n)|_n\colon n\in\gw\rangle\in b$ holds;
		\item $\bar\gs\vareps\bar\tau$ if $\langle |\bar\gs(n)\in\bar\tau(n)|_n\colon n\in\gw\rangle\in b$ holds.
	\end{itemize}
	
	\noindent It is not difficult to see that $\approx$ is an equivalence relation ($G\subset\mathbb{B}_0$ is a filter, after all) and $\vareps$ respects the equivalence relation $\approx$ (for the same reason). Now, in $V[G]$, consider the class map $\pi$ with domain $C/\approx$ defined by $\pi([\bar\gs]_\approx)=$the eventual value of $\bar\gs(n)/G$. The definition of $\approx$ shows that the values of $\pi$ do not depend on the choice of representatives of equivalence classes. Similarly, the map $\pi$ transports $\vareps$ to $\in$. It is immediately clear from the definition of the class $C$ that the values of $\pi$ belong to the intersection model. Also, every set $x\in M_\gw$ is in the range of $\pi$. To see this, consider the function $\bar \gs$ defined by $\bar\gs(n)=$the first $\mathbb{B}_n$-name such that $\bar\gs(n)/G=x$ in some fixed well-ordering of a large part of the ground model. The sequence $\bar\gs$ belongs to $M$ since for every $n\in\gw$, $\bar\gs\restriction [n, \gw)\in V[G\cap \mathbb{B}_n]$ holds. The sequence $\bar\gs$ then belongs to $V(a)$ by the first paragraph of the present proof. So, it belongs to $C$, and $\pi(\bar\gs)=x$ by the definitions.
	
	Now, it follows that $\vareps$ is a set-like, extensional, and well-founded relation on the class $C/\approx$, since it is isomorphic to one. By the Mostowski collapse lemma applied in the model $V(a)[b]$, it has a unique transitive collapse. This collapse must then be $\pi$. So $\pi$ is a class in $V(a)[b]$ and its range $M$ is a subset of $V(a)[b]$ as desired.
\end{proof}

\noindent In the key technical contribution of this paper, I will show that ZFC style preservation properties of the algebra $\mathbb{B}_0$ imply fragments of AC in the intersection models no matter the nature of the descending sequence of complete subalgebras.

\begin{proposition}
	If the poset $\mathbb{B}_0$ is proper, then in the intersection model every set can be injected into the product of an ordinal with the quotient $\mathbb{E}_0$ space.
\end{proposition}

\begin{proof}
	Let $A\in M$ be a set. Let $\preceq$ in $V$ be a well-ordering of a large fragment of the ground model. Let $f$ be the function with domain $A$ which assigns to any element $a\in A$ the set of all sequences $\bar\gs\in M$ such that for all but finitely many $n$, $\bar\gs(n)$ is the $\preceq$-first $\mathbb{B}_n$-name in the ground model such that $\bar\gs(n)/G=a$. It is clear that this function belongs to the model $V[G\cap\mathbb{B}_n]$ for every $n\in\gw$, so it belongs to the intersection model.
	
	Now, the properness of $\mathbb{B}_0$ implies that in $V[G]$, every countable subset of the ground model is a subset of a ground model set which is countable there. Thus, in $M$ one can form a function $g$ with domain $A$ such that for every $a\in A$, $g(a)$ is the $\preceq$-first set countable in $V$ which contains a tail of values of the sequences in $f(a)$. Now, let $h$ be the function with domain $A$ which assigns to each $a\in A$ the pair $\langle g(a), u\rangle$ where $u$ is the $\mathbb{E}_0$-class in $\gw^\gw$ which is mapped to the set $g(a)\cap h(a)^\gw$ by the $\preceq$-first bijection between $\gw$ and $g(a)$. It is immediate that $h$ is an injection of $A$ into the product of a subset of the ground model (which is well-orderable) and the $\mathbb{E}_0$ space.
\end{proof}

\begin{proposition}
	If the poset $\mathbb{B}_0$ is proper and bounding, then in the intersection model, every set of nonempty sets admits a multi-selector with countable values.
\end{proposition}

\begin{proof}
	First use the proper and bounding assumption to show that in the model $V[G]$, for every countable sequence $\bar s$ of ground model sets there is a ground model sequence $\bar t$ of finite sets such that $\forall n\ \bar s(n)\in\bar t(n)$.
	
	Now, suppose that $A\in M$ is a set of nonempty sets. Consider the function $f$ with domain $A$ such that for every $a\in A$, the functional value $f(a)$ is a sequence $\bar t$ in the ground model which takes only finite values, such that for all but finitely many $n\in\gw$, $\bar t(n)$ contains a $\mathbb{B}_n$-name $\gs$ such that $\gs/G\in A$, and such a sequence $\bar t$ is selected first in some fixed well-ordering of a large part of the ground model.
	
	Observe that the function $f$ is well-defined. If $a\in A$ is any element and $x\in a$ is arbitrary, in $V[G]$ one can find a sequence $\bar s$ such that for every $n\in\gw$, $\bar s(n)$ is a $\mathbb{B}_n$-name in the ground model such that $\bar s(n)/G=x$. Then one can approximate $\bar s$ with $\bar t$ in the ground model as indicated in the first paragraph of the present proof, showing that the definition of $f$ minimizes over a nonempty set.
	
	Now, let $g$ be the function on $A$ defined by $g(a)=\{x\in a\colon$ for all but finitely many $n\in\gw$ there is an element $\gs\in f(a)(n)$ which is a $\mathbb{B}_n$-name and $\gs/G=x\}$. It is immediate that the function $g$ is defined in the same way in every model $V[G\cap\mathbb{B}_n]$, so it belongs to the intersection model. By the definition of $f$, for each $a\in A$, $g(a)\subset a$ is a nonempty set. I must argue that the set $g(a)$ is countable in the intersection model.
	
	In order to do that, fix the set $a\in A$. Consider the sequence $\bar u$ defined by $\bar u(n)=\{x\in a\colon$ for every $m\geq n$ there is $\gs\in f(a)(n)$ which is a ground model $\mathbb{B}_n$-name and $\gs/B_n=x\}$. The sequence $\bar u$ consists of finite subsets of $g(a)$, and $g(a)$ is the increasing union $\bigcup_n\bar u(n)$. In addition, for each $n\in\gw$ the tail $\bar u\restriction [n, \gw)$ belongs to $V[G\cap\mathbb{B}_n]$, so $\bar u$ belongs to the intersection model. It follows that in $M$, $g(a)$ is an increasing union of finite sets, and by DC in $M$ it follows that such sets are countable.
\end{proof}

\begin{proposition}
	\label{bpiproposition}
 Suppose that $\mathbb{B}_0$ preserves a nonprincipal ultrafilter on $\gw$. Then BPI holds in the intersection model.	
\end{proposition}

\noindent Note that the proposition says that if the largest generic extension is not far from the ground model, all intermediate intersection models satisfy BPI. This observation together with Proposition~\ref{dcproposition} provides an endless supply of heretofore unknown models of ZF+DC+BPI.

\begin{proof}
	Let $C$ be a Boolean algebra in the intersection model; I must find an ultrafilter on $C$ in the intersection model. As in Proposition~\ref{dcproposition}, I construct  a sequence $\langle F_n\colon n\in\gw\rangle$ of ultrafilters on $C$ such that for every $m\in\gw$, $\langle F_n\colon n\geq m\rangle\in V[G\cap\mathbb{B}_m]$ holds.
	
	This is easy to do. Let $\preceq$ in $V$ be a well-ordering of a large part of the ground model. For each $n\in\gw$, write $F_n=\tau_n/G$ for the $\preceq$-first $\mathbb{B}_n$-name $\tau_n$ such that $\tau_n/G$ is an ultrafilter on $C$. Note that the sequence $\langle F_n\colon n\in\gw\rangle$ is as required.
	
	Now, let $U\in V$ be a nonprincipal ultrafilter on $\gw$ which generates an ultrafilter in $V[G]$; by an abuse of notation, let $U$ also denote the ultrafilter generated by $U$ in $V[G\cap\mathbb{B}_n]$ for $n\in\gw$. Let $F=\{b\in B\colon\{n\in\gw\colon b\in F_n\}\in U\}$. Clearly, this is an ultrafilter on $C$; I must prove that for every $m\in\gw$, $F\in V[G\cap\mathbb{B}_m]$. This is immediate though. $F$ is recovered in $V[G\cap\mathbb{B}_m]$ as $F=\{b\in B\colon\{n\geq m\colon b\in F_n\}\in U\}$, using the fact that both $U$ and $\langle F_n\colon n\geq m\rangle$ belong to $V[G\cap\mathbb{B}_m]$.
\end{proof}

\section{General theorems on $\mathbb{E}_0$ forcing}

In this section, I recall the basic facts about the generic objects used to form the intersection model for Theorem~\ref{maintheorem}. None of them are new.

The $\mathbb{E}_0$-forcing has appeared in the literature many times \cite[Section 4.7.1]{z:book2}. In this section I describe its well-known combinatorial presentation and state an ultrafilter preservation theorem needed for the final proof.

\begin{definition}
	An $\mathbb{E}_0$-tree is an ever-branching tree $T\subset 2^{<\gw}$ such that for every splitnode $t\in T$ and every $u\in T$ extending $t$, the node obtaining from $u$ by flipping its $|t|$-th entry is in $T$ as well.
\end{definition}

\noindent Let $I$ be the $\gs$-ideal on $\cantor$ $\gs$-generated by Borel partial $\mathbb{E}_0$-selectors. The (proof of the) Glimm-Effros dichotomy of \cite{hkl:glimmeffros} immediately implies that a Borel subset of $\cantor$ is $I$-positive if and only if it contains all branches of some $\mathbb{E}_0$-tree. Thus, the poset of Borel sets ordered by inclusion modulo $I$ is in the forcing sense equivalent to the poset of $\mathbb{E}_0$-trees ordered by inclusion, and this is the combinatorial presentation of the poset I use below.

\begin{fact}
	\label{ufact}
Let $U$ be a Ramsey ultrafilter on $\gw$. Let $\mathbb{P}$ be a countable support product of any number of $\mathbb{E}_0$-forcings. In the $\mathbb{P}$-extension, $U$ still generates an ultrafilter.	
\end{fact}

\begin{proof}
	This is a concatenation of several arguments well-known from pre-existing literature; I will only provide an outline of the proof. Suppose that $p\in\mathbb{P}$ is a condition and $p\Vdash\gs\subset\gw$ is a set. I must show that there is a condition $r\leq p$ and a set $a\in U$ such that $r\Vdash\check a\subseteq\gs$ or $r\Vdash\check a\cap\gs=0$. 
	
	Let $I$ denote the index set for the product $\mathbb{P}$. A standard fusion argument provides a strengthening $q\leq p$ and a countable set $J\subset I$ such that below $q$, $\gs$ is in fact a name in the product $\mathbb{Q}$ of $J$-many copies of the $\mathbb{E}_0$-forcing. The poset $\mathbb{Q}$ adds no independent reals by \cite[Example 1.10]{z:subadditive}; so the set $A=\{a\subset\gw\colon\exists r\leq q\ r\Vdash\check a\subseteq\gs$ or $r\Vdash\check a\cap\gs=0\}$ is dense in the algebra $\power(\gw)$ modulo finite. Use the continuous reading of names to strengthen $q$ if necessary to make sure that the set $A\subset\power(\gw)$ is analytic. Every Ramsey ultrafilter has a nonempty intersection with every analytic set dense in $\power(\gw)$ modulo finite by the arguments of \cite{mathias:happy}. This provides the desired condition $r$ and the set $a\in U$. 
\end{proof}

\section{Conclusion of the argument}

At this point, everything is ready to conclude the proof of Theorem~\ref{maintheorem}. Suppose that the Continuum Hypothesis holds. Then there is a Ramsey ultrafilter on $\gw$, pick one and call it $U$. Let $\mathbb{P}$ be a countable support product of $\gw_2$-many copies of $\mathbb{E}_0$-forcing. Let $\dot c$ be the $\mathbb{P}$-name for the generic object, viewed as a map from $\gw_2\times\gw$ to $2$. Let $\mathbb{B}_0$ be the Boolean completion of $\mathbb{P}$, and for each $n\in\gw$ let $\mathbb{B}_n$ be the complete subalgebra of $\mathbb{B}_0$ generated by the names $\dot c(\check\ga, \check m)$ for $m\geq n$.

Starting from a model of ZFC+CH, let $c\colon\gw_2\times\gw\to 2$ be a $\mathbb{P}$-generic function. For notational convenience, for every ordinal $\ga$ write $c(\ga)\in\cantor$ for the $\ga$-th section of the map $c$, and for every $n\in\gw$ write $c_n$ for the function from $\gw_2\times\gw$ to $2$ defined by $c_n(\ga, m)=c(\ga, n+m)$. It is immediate that the function $c_n$ is just a shift of the $\mathbb{B}_n$-generic object. The intersection model satisfies ZF+DC by Proposition~\ref{dcproposition}. The ultrafilter $U$ generates an ultrafilter in $V[c]$ by Fact~\ref{ufact}. Proposition~\ref{bpiproposition} then shows that the intersection model satisfies BPI.

It remains to show that there is no Vitali set in the intersection model. In fact, observe that the sequence $\langle [c(\ga)]_{\mathbb{E}_0}\colon \ga\in\gw_2\rangle$ belongs to the intersection model; I will show that there is no selector on this sequence of sets. This is a consequence of a rather standard $\Delta$-system argument in the ground model. To state it, I need to be able to move freely between the generic function and its fragments on the coherent sequence. The following definition is designed to facilitate such moves.

\begin{definition}
	Suppose that $n\in\gw$ is a number. 
	
	\begin{enumerate}
		\item For a function $c\colon \gw_2\times\gw\to 2$, the function $c|n\colon\gw_2\times\gw\to 2$ is defined by $(c|n)(\ga, m)=c(\ga, n+m)$.
		\item For a condition $p\in\mathbb{P}$ and a function $v\colon\gw\to 2^n$ such that for every $\ga\in\gw_2$, $v(\ga)\in p(\ga)$ holds, write $p|v$ for the condition defined by the demand $w\in (p|v)(\ga)$ iff $v(\ga)^\smallfrown w\in p(\ga)$.
		\item For a condition $p\in\mathbb{P}$ and a function $v\colon\supp(p)\to\ga$, write $p||v$ for the condition in $\mathbb{P}$ with the same domain defined by the demand $w\in (p||v)(\ga)$ iff $w$ is an initial segment of $v(\ga)$ or there is $u\in p(\ga)$ such that $w=v(\ga)^u$.
	\end{enumerate}
\end{definition}

\begin{proposition}
	Let $n\in\gw$ be a number. The poset $\mathbb{P}$ forces the function $\dot c|n$ to be $\mathbb{P}$-generic over the ground model.
\end{proposition}

\begin{proof}
	Consider the set $D$ of all conditions $p$ such that for every $\ga\in\supp(p)$, the stem of $p(\ga)$ is longer than $n$. It is obvious that the set $D\subset P$ is dense. Let $\pi\colon D\to\mathbb{P}$ be the function defined by $\pi(p)=p|v$ for some (any) function $v\colon \gw_2\to 2^n$ such that for every $\ga\in\supp(p)$, $v(\ga)\in p(\ga)$. It is not difficult to check that $\mathbb{P}$ forces the $\pi$-image of $D\cap\dot G$ to be a generic filter on $\mathbb{P}$ which yields the function $\dot c|n$.
\end{proof}

I need to show that there is no Vitali set in the intersection model. Suppose towards a contradiction that some condition $p\in\mathbb{P}$ forces $\gs\subset\cantor$ to be a set in the intersection model with singleton intersection with all the sets $[\dot c(\ga)]_{\mathbb{E}_0}$ for $\ga\in\gw_2$. For each ordinal $\ga\in\gw_2$, there must be a condition $p_\ga\leq p$ and a number $m_\ga\in\gw$ such that $p_\ga$ forces the unique element of $\gs$ in the $\mathbb{E}_0$-class of $\dot c(\ga)$ to agree with $\dot c(\ga)$ past $\check m_\ga$. There must be a number $n_\ga\in\gw$ such that the tree $p_\ga(\ga)$ contains two splitnodes at the level $n_\ga$ which differ at some entry larger than $m_\ga$. By the CH assumption, there must be a cofinal set $A\subset\gw_2$ such that the numbers $m_\ga$ for all $\ga\in A$ coincide and are equal to some $m$, the numbers $n_\ga$ for all $\ga\in A$ coincide and are equal to some $n$, the sets $\dom(p_\ga)$ for $\ga\in A$ form a $\Delta$-system with heart $h\subset\gw$, and the conditions $p_\ga\restriction h$ all coincide and are equal to some $q\in\mathbb{P}$.

Let $r\leq q$ be a condition and $\tau$ be a $\mathbb{P}$-name such that $r\Vdash\gs/\dot c=\tau/\dot c|\check n$. Choose an ordinal $\ga\in A$ so that $\dom(p_\ga)\cap\dom(r)=h$; so, $p_\ga$ is compatible with $r$. Let $r_\ga$ be the specific common lower bound of them defined by $\supp(r_\ga)=\supp(r)\cup\supp(p_\ga)$, for each $\gb\in\supp(r)$ $r_\ga(\gb)=r(\gb)$, and for each $\gb\in\supp(p_\ga)\setminus h$ $r_\ga(\gb)=p_\ga(\gb)$. Choose functions $v_0, v_1\colon \gw_2\to 2^n$ such that they differ only on $\alpha$-th entry, for each $\gb\neq\ga$ $v_0(\gb)=v_1(\gb)\in r_\ga(gb)$, and the nodes $v_0(\ga), v_1(\ga)\in p_\ga(\ga)$ disagree at some entry larger than $m$. Note that the conditions $r_\ga|v_0$ and $r_\ga|v_1$ coincide by the symmetry properties of the $\mathbb{E}_0$-tree $p_\ga(\ga)$.

Find a condition $s\leq r_\ga|v_0$ and a string $w\in 2^n$ such that $s$ forces the unique element of $\tau$ which is $\mathbb{E}_0$-equivalent to $(0^n)^\smallfrown\dot c_\ga$ to start with $\check w$. By the choice of $v_0$ and $v_1$, one of the strings $v_0(\ga),v_1(\ga)\in 2^n$ disagrees with $w$ at an entry $k>m$; for definiteness, assume $v_0(\ga)(k)=0$ while $w(k)=1$. Consider the condition $s||v_0\in\mathbb{P}_0$. It is stronger than $p_\ga$, so it forces the unique element of $\gs$ in the $\mathbb{E}_0$-class of $\dot c(\ga)$ to have $0$ at $m$-th entry. However, by the choice of $s$ it also forces that unique element to start with $w$. This is a contradiction, showing that no Vitali sets can exist in the intersection model.

\bibliographystyle{plain} 
\bibliography{odkazy,zapletal}

\end{document}